\documentclass[11pt,a4paper]{article}
\usepackage[utf8]{inputenc}
\usepackage[T1]{fontenc}
\usepackage{amsmath,amssymb,amsthm}
\usepackage{graphicx}
\usepackage{geometry}
\usepackage{hyperref}
\newtheorem{theorem}{Theorem}

\theoremstyle{remark}
\newtheorem{remark}[theorem]{Remark}
\newcommand{\R}{\mathbb{R}}
\newcommand{\inner}[2]{\langle #1, #2 \rangle}
\DeclareMathOperator{\zer}{zer}

\title{\textbf{The sharp step-size constant for one-call reflection splittings on monotone inclusions}}
\author{Yekini Shehu\\ \small School of Mathematical Sciences, Zhejiang Normal University,\\ \small Jinhua 321004, China; \texttt{yekini.shehu@zjnu.edu.cn}}
\date{}
\begin{document}
\maketitle

\begin{abstract}
\noindent 
The forward-reflected-backward splitting of Malitsky and Tam converges weakly
for step sizes $\lambda\in(0,\tfrac{1}{2L})$, where $B$ is monotone and
$L$-Lipschitz, and the question
whether this bound is tight has been recorded as open \cite{GS}. We show that it is: for the matched-skew instance $A=LJ$, $B=LJ$
($J$ the counterclockwise rotation by $\pi/2$ in $\R^2$), the iterates fail
to converge for every $\lambda\ge\tfrac{1}{2L}$ and diverge for
$\lambda>\tfrac{1}{2L}$. More generally, for the skew--rotation family
$A=\gamma J$, $B=J$, the exact stability threshold is
$\lambda^\star(\gamma)=1/\sqrt{(1+\gamma)(3-\gamma)}$, which attains its
minimum $1/2$ at the matched skew $\gamma=1$ and recovers the reflected
gradient constant $1/\sqrt3$ at $\gamma=0$. The same instance is a
counterexample for the reflected--forward--backward method of Cevher and V\~u,
which coincides with forward-reflected-backward on linear operators.


\end{abstract}

\section{Result}
Let $\mathcal H$ be a real Hilbert space, $A\colon\mathcal H\rightrightarrows\mathcal H$
maximally monotone, and $B\colon\mathcal H\to\mathcal H$ monotone and
$L$-Lipschitz with $\zer(A+B)\neq\varnothing$. The \emph{forward-reflected-backward}
(FRB) method of Malitsky and Tam \cite{MT} ---
extending the reflected gradient method of \cite{M15} from variational
inequalities to inclusions --- generates $x_{k+1}$ from
\begin{equation}\label{eq:frb}
x_{k+1}=J_{\lambda A}\bigl(x_k-2\lambda B(x_k)+\lambda B(x_{k-1})\bigr),
\end{equation}
using one resolvent evaluation and one forward evaluation per iteration, and
converges weakly to a zero of $A+B$ for every fixed
$\lambda\in(0,\tfrac{1}{2L})$ \cite[Theorem~2.5]{MT}. The
\emph{reflected--forward--backward} (RFB) method of Cevher and V\~u \cite{CV}
replaces the reflected forward step by a reflected evaluation point,
\begin{equation}\label{eq:rfb}
x_{k+1}=J_{\lambda A}\bigl(x_k-\lambda B(2x_k-x_{k-1})\bigr),
\end{equation}
and likewise converges under the sole monotonicity--Lipschitz
assumption \cite{CV}, with a certified step-size range strictly below the
sharp constant established here. On affine operators $B$ the two schemes coincide, since
$B(2x_k-x_{k-1})=2B(x_k)-B(x_{k-1})$.

\begin{theorem}[sharpness of the FRB/RFB step-size bound]\label{thm:main}
Let $J\colon\R^2\to\R^2$ denote the rotation $J(u,v)=(-v,u)$, let $L>0$, and
set $A=LJ$ and $B=LJ$ (skew-adjoint linear operators are maximally monotone
\cite{BC}). Then $B$ is monotone and $L$-Lipschitz, and
$\zer(A+B)=\{0\}$. For every $\lambda\ge\tfrac{1}{2L}$ and
generic initial data $(x_0,x_{-1})\in\R^2\times\R^2$ the sequences generated
by \eqref{eq:frb} and \eqref{eq:rfb} do not converge to the solution; for
$\lambda>\tfrac{1}{2L}$ they diverge geometrically. Consequently
$\tfrac{1}{2L}$ is the exact uniform step-size constant for the class: no
larger constant can be certified for all maximal monotone $A$ and monotone
$L$-Lipschitz $B$.
\end{theorem}

\begin{proof}
Identify $\R^2$ with $\mathbb C$ so that $J$ is multiplication by $i$. Both
schemes reduce to the scalar recursion
\[
x_{k+1}=\frac{(1-2i\tilde\lambda)x_k+i\tilde\lambda\,x_{k-1}}{1+i\tilde\lambda},
\qquad \tilde\lambda:=\lambda L>0,
\]
with state matrix $\mathcal A(\tilde\lambda)$ whose characteristic equation is
\begin{equation}\label{eq:char}
(1+i\tilde\lambda)\,r^2-(1-2i\tilde\lambda)\,r-i\tilde\lambda=0.
\end{equation}
We locate the boundary of the stability region. Substituting $r=e^{i\varphi}$,
$c=\cos\varphi$, $s=\sin\varphi$, dividing \eqref{eq:char} by $e^{i\varphi}$,
and separating real and imaginary parts gives
\[
c=1+\tilde\lambda s(1+\gamma)\ \ \text{at}\ \gamma=1:\quad
c=1+2\tilde\lambda s,\qquad s=-2\tilde\lambda .
\]
(For general $\gamma$ see Theorem~\ref{thm:curve}.)
Imposing $c^2+s^2=1$ yields $16\tilde\lambda^4-4\tilde\lambda^2=0$, whose
unique positive solution is $\tilde\lambda=\tfrac12$. At
$\tilde\lambda=\tfrac12$ the roots of \eqref{eq:char} are
\[
r_1=-i,\qquad r_2=\frac{2-i}{5},\qquad |r_1|=1,\ \ |r_2|=\frac1{\sqrt5}<1.
\]
The root $r_1$ is simple, and implicit differentiation of \eqref{eq:char}
gives
\[
\frac{d|r_1|^2}{d\tilde\lambda}\Big|_{\tilde\lambda=1/2}
=2\operatorname{Re}\Bigl(\overline{r}_1\,\frac{dr_1}{d\tilde\lambda}\Bigr)=4>0,
\]
so $|r_1|>1$ for all $\tilde\lambda>\tfrac12$: the state matrix has spectral
radius larger than one, and since the eigenvalue branch is continuous and
simple, generic initial data excite the expanding mode and
$\|x_k\|\to\infty$ geometrically. At $\tilde\lambda=\tfrac12$ the persistent
mode $c(-i)^k$ with $c\neq0$ (generic data) does not converge to $0$. Hence
no convergence for $\tilde\lambda\ge\tfrac12$ and geometric divergence for
$\tilde\lambda>\tfrac12$. Together with the Malitsky--Tam theorem
\cite[Theorem~2.5]{MT}, which covers every $\tilde\lambda\in(0,\tfrac12)$ for
every instance of the class, this proves that $\tfrac{1}{2L}$ is the exact
uniform constant.
\end{proof}

\begin{theorem}[exact threshold for the skew--rotation family]\label{thm:curve}
Let $A=\gamma J$, $B=J$ ($\gamma\ge0$). The iterates of \eqref{eq:frb}
converge to the unique solution $x_\star=0$ for every
$\lambda\in(0,\lambda^\star(\gamma))$ and diverge for
$\lambda>\lambda^\star(\gamma)$, where
\begin{equation}\label{eq:curve}
\lambda^\star(\gamma)=\frac{1}{L\sqrt{(1+\gamma)(3-\gamma)}}\quad(0\le\gamma<3),
\qquad \lambda^\star(\gamma)=+\infty\quad(\gamma\ge 3),
\end{equation}
with the convention $L=1$ for $B=J$. In particular $\lambda^\star(0)=1/\sqrt3$
(the reflected-gradient constant \cite[Example~10.1]{S}), the minimum over
$\gamma\ge0$ is $\lambda^\star(1)=\tfrac12$, and $\lambda^\star(2)=1/\sqrt3$.
\end{theorem}

\begin{proof}
With $\tilde\lambda=\lambda$ ($L=1$) the characteristic equation is
$(1+i\gamma\tilde\lambda)r^2-(1-2i\tilde\lambda)r-i\tilde\lambda=0$. The
boundary computation of Theorem~\ref{thm:main}, performed for general
$\gamma$, gives
\[
c=\frac{1-2(1+\gamma)\tilde\lambda^2}{1+(\gamma^2-1)\tilde\lambda^2},
\qquad
s=\frac{-(1+\gamma)\tilde\lambda}{1+(\gamma^2-1)\tilde\lambda^2},
\]
and $c^2+s^2=1$ simplifies to
\[
\tilde\lambda^2(1+\gamma)^2\bigl((3-\gamma)(1+\gamma)\tilde\lambda^2-1\bigr)=0,
\]
whose unique positive solution is
$\tilde\lambda=1/\sqrt{(1+\gamma)(3-\gamma)}$ for $0\le\gamma<3$ (none for
$\gamma\ge3$). Hence no root lies on the unit circle for
$\tilde\lambda\neq\lambda^\star(\gamma)$. Both roots are inside the disk for
small $\tilde\lambda$: expanding the quadratic formula,
$|r_{\max}|^2=1-(1+\gamma)^2\tilde\lambda^2+O(\tilde\lambda^3)<1$. By
continuity of the roots in $\tilde\lambda$ and the uniqueness of the boundary
point, convergence holds on $(0,\lambda^\star(\gamma))$. At the boundary the
on-circle root is $r_1=c^\star+is^\star$ with
\[
c^\star=\frac{1-\gamma}{2},\qquad
s^\star=-\frac{\sqrt{(1+\gamma)(3-\gamma)}}{2},
\]
the second root has modulus
$|r_2|=\tilde\lambda^\star/\sqrt{1+\gamma^2\tilde\lambda^{\star2}}<1$, and
implicit differentiation gives
$\frac{d|r_1|^2}{d\tilde\lambda}\big|_{\lambda^\star}>0$ (e.g.\ $4$ at
$\gamma=1$, $3\sqrt3$ at $\gamma=0$), so the crossing is transversal and
divergence holds for $\tilde\lambda>\lambda^\star(\gamma)$. The values
$\lambda^\star(0)=1/\sqrt3$ and $\lambda^\star(1)=1/2$ are immediate;
$(1+\gamma)(3-\gamma)\le4$ on $[0,3]$ with equality at $\gamma=1$, giving the
minimum.
\end{proof}

\begin{remark}[landscape]\label{rem:landscape}
Three observations frame the general program. (i) \emph{The skew part of $A$ is the obstruction, not the symmetric part}: Theorem~\ref{thm:phase} makes this exact and global --- $\lambda^\star(a,\gamma)$ is the first positive root of an explicit cubic, the global minimum over $a\ge0$, $\gamma\in\R$ is attained at the matched skew $(0,1)$ (with the convention $\lambda^\star(0,-1)=+\infty$ at the degenerate point), and the symmetric part of $A$ cannot lower the threshold.  The nonlinear analogue of the phase diagram is open.
(ii) \emph{Why point-reflection resists quadratic Lyapunov budgets}: the
proof technique of \cite{S} (exact dissipation budgets) relies on the
projection identity in its Step~2; for a general maximal monotone $A$ the
available substitute is monotonicity of $A$ across consecutive iterates,
which yields
$\inner{d_k-d_{k+1}}{d_{k+1}}\le\lambda\inner{B(u_k)-B(u_{k-1})}{d_{k+1}}$ ---
a pairing with the wrong sign structure for that framework. This explains why
the certifiable one-call inclusion schemes are value-reflected \eqref{eq:frb}
rather than point-reflected \eqref{eq:rfb}, although \eqref{eq:rfb} converges
\cite{CV} and coincides with \eqref{eq:frb} on linear instances.
(iii) The existing tightness results concern the \emph{rate} at $A=0$
($1/\sqrt2$ for GFRB on the rotation \cite{GFRB}); the \emph{threshold}
tightness settled here is, to our knowledge, new.
\end{remark}

\begin{figure}[t]
\centering
\includegraphics[width=\textwidth]{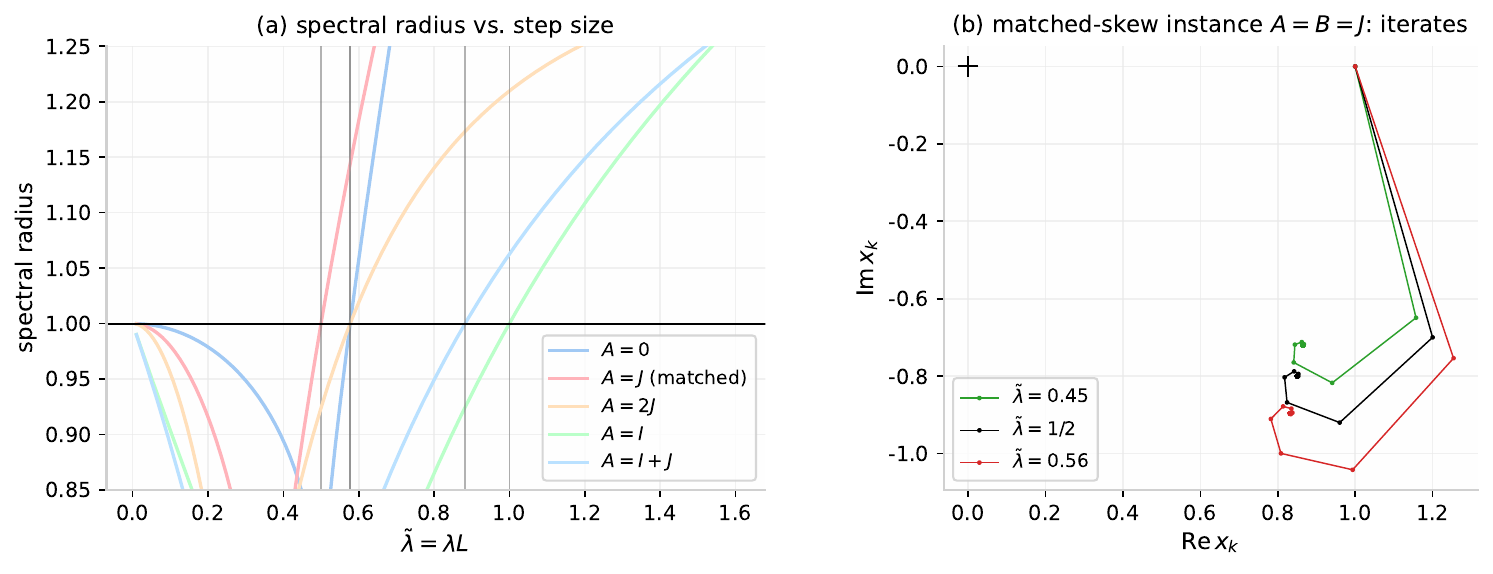}
\caption{(a) Spectral radius of the FRB state matrix for representative
instances $A=aI+\gamma J$, $B=J$ ($L=1$): the curves cross $1$ exactly at
the thresholds given by Theorems~\ref{thm:curve} and~\ref{thm:phase}
(vertical lines), and the matched skew $A=J$ has the earliest crossing at
$\tilde\lambda=\tfrac12$.  (b) Iterates of the matched-skew instance for
$\tilde\lambda=0.45$ (convergent spiral), $\tilde\lambda=\tfrac12$
(persistent rotation; Theorem~\ref{thm:main}), and $\tilde\lambda=0.56$
(divergent spiral).}
\label{fig:spectral}
\end{figure}

\section{The phase diagram}

\begin{figure}[t]
\centering
\includegraphics[width=0.8\textwidth]{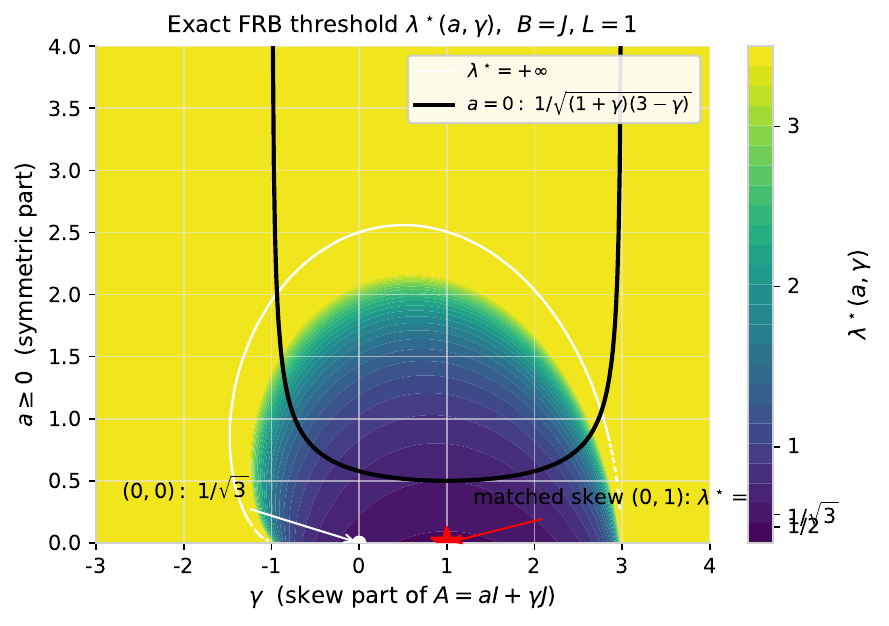}
\caption{Phase diagram of the exact threshold $\lambda^\star(a,\gamma)$
(Theorem~\ref{thm:phase}) for $A=aI+\gamma J$, $B=J$ ($L=1$), computed from
the cubic of Theorem~\ref{thm:phase} over a $400\times320$ grid; the gray
(white-outline) region is $\lambda^\star=+\infty$.  The black curve is
the analytic section $a=0$ of Theorem~\ref{thm:curve}; the star marks the
global worst case $(0,1)$ with $\lambda^\star=\tfrac12$, and the dot
marks $(0,0)$ with $\lambda^\star=1/\sqrt3$.  The symmetric part $a>0$
enlarges the threshold.}
\label{fig:phase}
\end{figure}

Theorem~\ref{thm:curve} is the section $a=0$ of a complete phase diagram for
the normal family $A=aI+\gamma J$, $B=J$.

\begin{theorem}[phase diagram]\label{thm:phase}
Let $A=aI+\gamma J$ with $a\ge0$, $\gamma\in\R$, and $B=J$.  Define the cubic
\[
P(\lambda)=A_3\lambda^3+A_2\lambda^2+A_1\lambda+A_0,\qquad
\begin{aligned}
A_3&=a^4+2a^2\gamma^2-6a^2+\gamma^4-6\gamma^2-8\gamma-3,\\
A_2&=4a^3+4a\gamma^2-4a,\\
A_1&=5a^2+\gamma^2+2\gamma+1,\\
A_0&=2a ,
\end{aligned}
\]
and let $\lambda^\star(a,\gamma)$ be the smallest positive root of $P$, with
$\lambda^\star(a,\gamma)=+\infty$ if none exists (at the single point
$(a,\gamma)=(0,-1)$ the cubic vanishes identically and $\lambda^\star=+\infty$ by
the analysis below).  The iterates of \eqref{eq:frb} converge to an element
of $\zer(A+B)$ for every $\lambda\in(0,\lambda^\star(a,\gamma))$ and diverge for
every $\lambda>\lambda^\star(a,\gamma)$.  Here $\zer(A+B)=\{0\}$ except at
$(a,\gamma)=(0,-1)$, where $A+B\equiv0$ and $\zer(A+B)=\R^2$: there $r=1$ is a
simple root of the characteristic equation for every $\lambda>0$ with the
second root of modulus $\lambda/\sqrt{1+\lambda^2}<1$, so the iterates converge
to $\zer(A+B)$ for every $\lambda>0$.  Moreover
\begin{equation}\label{eq:min}
\min_{\substack{a\ge0\\\gamma\in\R}}\lambda^\star(a,\gamma)=\lambda^\star(0,1)=\tfrac12 ,
\end{equation}
so the matched skew $(a,\gamma)=(0,1)$ is the global worst case of the entire
family, and the symmetric part $a>0$ cannot lower the global threshold.
\end{theorem}

\begin{proof}
The boundary computation of Theorem~\ref{thm:curve} applied to
$(1+\lambda(a+i\gamma))r^2=(1-2i\lambda)r+i\lambda$ gives the linear system
\[
(1+\lambda a)c-\lambda\gamma s=1+\lambda s,\qquad
(1+\lambda a)s+\lambda\gamma c=\lambda c-2\lambda ,
\]
whose solution is
\[
c=\frac{1+a\lambda-2(1+\gamma)\lambda^2}{(1+a\lambda)^2+(\gamma^2-1)\lambda^2},
\qquad
s=\frac{-\lambda\,(1+\gamma+2a\lambda)}{(1+a\lambda)^2+(\gamma^2-1)\lambda^2},
\]
and a direct expansion gives
\[
c^2+s^2-1=\frac{-\lambda\,P(\lambda)}{\bigl((1+a\lambda)^2+(\gamma^2-1)\lambda^2\bigr)^2}.
\]
Hence every boundary point is a positive root of $P$, and the smallest
root $\lambda^\star$ is the first boundary.  Both
roots are inside the disk for small $\lambda$: for $a>0$,
$|r_{\max}|^2=1-2a\lambda+O(\lambda^2)<1$, while for $a=0$ the sharper
$|r_{\max}|^2=1-(1+\gamma)^2\lambda^2+O(\lambda^3)<1$ applies (the expansion
remains valid for $-1<\gamma<0$).  Continuity of the roots then gives
convergence on $(0,\lambda^\star)$.  At $\lambda^\star$ the on-circle root is simple and
the second root lies strictly inside; implicit differentiation gives
$\frac{d|r_1|^2}{d\lambda}\big|_{\lambda^\star}>0$ (exactly $4$ at the
worst case $(0,1)$, $3\sqrt3$ at $(0,0)$), so the crossing is transversal
and divergence holds beyond it.

For \eqref{eq:min}, let $(c,s)$ correspond to any boundary point with
$\lambda<\tfrac12$.  Multiplying the two equations of the boundary system
by $c$ and $s$, adding, and using $c^2+s^2=1$ gives
$1+\lambda a=c-2\lambda s(1-c)\le c+2\lambda(1-c)=2\lambda+c(1-2\lambda)\le
2\lambda+(1-2\lambda)=1$, forcing $a=0$, $c=1$,
$s=0$; the second boundary equation at this point requires $\gamma=-1$, which
is precisely the degenerate case analysed in the statement, where
$\lambda^\star(0,-1)=+\infty$ and no crossing of the stability boundary
occurs.  For every $\gamma\neq-1$ the second equation is violated --- a
contradiction.  Hence no boundary point exists with $\lambda<\tfrac12$ for any
$a\ge0$,
$\gamma\neq-1$; $\lambda^\star(0,1)=\tfrac12$ by
Theorem~\ref{thm:curve}; and the stated examples
($\lambda^\star(1,0)=1$ from $P=-2(\lambda-1)(2\lambda+1)^2$,
$\lambda^\star(0,2)=1/\sqrt3$) show the threshold is neither monotone in
$\gamma$ nor confined to $[0,\tfrac12]$.
\end{proof}

\begin{remark}\label{rem:diag}
Three consequences. (i) The threshold is finite within the region where $P$
has a positive root and $+\infty$ outside it ($e.g.\ a=0$, $\gamma\ge3$, or
$\gamma\le-1$, where the method converges for every $\lambda>0$; at the
single point $(0,-1)$ this is the semisimple $r=1$ case with
$\zer(A+B)=\R^2$). (ii) Along $a=0$ one
recovers Theorem~\ref{thm:curve}; along $\gamma=0$ the cubic has no general
factorization (the case $a=1$ is solvable, $P=-2(\lambda-1)(2\lambda+1)^2$)
and the threshold grows with $a$. (iii) The dichotomy \emph{symmetric part
helps / skew part resonates} is now exact and global: \eqref{eq:min}
identifies the matched skew as the unique worst case, complementing the
sufficient condition of \cite{MT} from below.
\end{remark}

\end{document}